\documentclass[11pt,leqno]{article}

\usepackage{fullpage}
\usepackage{amssymb}
\usepackage{amsmath}
\usepackage{amsthm}
\newtheorem{theorem}{Theorem}

\newtheorem{lemma}{Lemma}

\begin{document}
\title{ On the Products $(1^\ell+1)(2^\ell+1)\cdots (n^\ell +1)$, II\footnote{This work was supported by the National Natural
Science Foundation of China, Grant No. 11371195. }}
\date{}
\author{Yong-Gao Chen\footnote{Corresponding author: ygchen@njnu.edu.cn} and Ming-Liang Gong
\\
\small School of Mathematical Sciences and Institute of
Mathematics, \\ \small Nanjing Normal
University, Nanjing 210023, P. R. CHINA}
 \maketitle
\begin{abstract} In this paper, the following results are proved: (i) For any odd integer $\ell$
 with at most two distinct prime factors and any positive integer $n$,
the product $(1^\ell+1)(2^\ell+1)\cdots (n^\ell +1)$ is not a
powerful number; (ii) For any integer $r\ge 1$, there exists a
positive integer $T_r$ such that, if $\ell$ is a positive odd
integer with at most $r$ distinct prime factors and $n$ is an
integer with $n\ge T_r$, then $(1^\ell+1)(2^\ell+1)\cdots (n^\ell
+1)$ is not a powerful number.
\end{abstract}



\section{ Introduction}

A positive integer $a$ is called a \emph{powerful number} if
$p\mid a$ implies $p^2\mid a$ for any prime $p$. Amdeberhan,
Medina and Moll \cite{Amdeberhan2008} conjectured that the product
$(1^2+1)(2^2+1)\cdots (n^2+1)$ is not a square for any integer
$n>3$. Cilleruelo \cite{Cilleruelo2008} confirmed this conjecture.
Fang \cite{Fang2009} confirmed another similar conjecture posed by
Amdeberhan, Medina and Moll \cite{Amdeberhan2008}. Amdeberhan,
Medina and Moll \cite{Amdeberhan2008} also claimed that if $n>12$
and $\ell $ is an odd prime,
 then the product $(1^\ell+1)(2^\ell+1)\cdots (n^\ell +1)$ is not a square. G\" urel and Kisisel \cite{Gurel2010}
 proved that $(1^3+1)(2^3+1)\cdots (n^3 +1)$ is not a powerful
 number. Zudilin, in his MathSciNet review No. MR2569849, outlined a
 proof of the claim for any prime $\ell \ge 5$.
Zhang and Wang \cite{ZhangWang} proved the claim for any prime
$\ell \ge 5$. Recently, the authors and Ren \cite{ChenGongRen}
proved that (i) For any odd prime power $\ell$ and any positive
integer $n$, the product $(1^\ell+1)(2^\ell+1)\cdots (n^\ell +1)$
is not a powerful number;
 (ii) For any  positive odd number $\ell$, there exists an integer $N_\ell$ such that for any positive integer $n\ge N_\ell$,
the product $(1^\ell+1)(2^\ell+1)\cdots (n^\ell +1)$ is not a
powerful number.

For any positive integer $\ell$, let
$$\Omega_l(n)=(1^\ell+1)(2^\ell+1)\cdots (n^\ell +1).$$

%´ýÐÞ¸Ä   l  to \ell  2013-10-6

In this paper, the following results are proved.

\begin{theorem}\label{thm1} For any positive odd integer $\ell$
 with at most two distinct prime factors and any positive integer $n$,
 $\Omega_\ell(n)$ is not a powerful number. \end{theorem}

\begin{theorem}\label{thm2}For any integer $r\ge 1$, there exists a positive integer $T_r$
such that, if $\ell$ is a positive odd integer with at most $r$
distinct prime factors and $n$ is an integer with $n\ge T_r$, then
$\Omega_\ell (n)$ is not a powerful number.
\end{theorem}

\section{General Lemmas}

In this section, we present some lemmas which can be applied to
any odd integer $\ell$.

Let $\cal P$ be the set of all positive primes. For a positive
integer $n$, let
$$P(n)=\{ p : p\in {\cal P}, \frac{n+1}2<p\le n+1\} $$
and
$$P(n;k,1)=\{ p : p\in P(n), p\equiv 1\pmod k\} .$$
For a prime $p$ and an integer $a$, $p^k\| a$ means that $p^k\mid
a$ and $p^{k+1}\nmid a$. For any real number $x$ and two coprime
positive integers $k$ and $\ell$, let $\pi (x)$ denote the number
of primes not exceeding $x$ and $\pi (x; k, \ell)$ denote the
number of primes not exceeding $x$ and congruent to $\ell$ modulo
$k$, and let
$$\vartheta (x; k, \ell) =\sum_{\substack{p\equiv \ell\hskip -3mm\pmod k\\ p\le x}}\log
p,$$ where $p$ denotes a prime. Let $\varphi (k)$ denote Euler's
totient function.

\begin{lemma}\label{lem1} Let $p$ be an odd prime and $a,\ell$ be two positive integers with $2\nmid \ell$. If
$$p\Big| \frac{a^\ell+1}{a+1},$$
then $(p(p-1), \ell )>1$. \end{lemma}

\begin{proof} Suppose that $(p(p-1), \ell )=1$. Then $(2\ell,p-1)=2$. By $p\mid a^\ell +1$,
we have $p\mid a^{2\ell}-1$ and $p \nmid a-1$. It follows from
Fermat's little theorem that $p\mid a^{p-1}-1$. Let $d$ be the
order of $a\pmod p$. Then $d\mid 2\ell$ and $d\mid p-1$ (see
\cite[Theorem 88]{Hardy}). So $d\mid (2\ell, p-1)$. Thus $p\mid
a^{(2\ell, p-1)}-1$. So $p\mid a^2-1$. Since $p\nmid a-1$, it
follows that $p\mid a+1$. Thus
$$0\equiv \frac{a^\ell +1}{a+1}=a^{\ell -1}-a^{\ell -2}+\cdots -a+1\equiv
\ell\pmod p,$$ a contradiction with $(p(p-1), \ell )=1$.
Therefore, $(p(p-1), \ell )>1$.
\end{proof}

\begin{lemma}\label{lem2} Let $n$ and $\ell$ be two positive integers with $2\nmid \ell$ such that $\Omega_\ell(n)$ is a powerful number. If
$p\in P(n)$,  then $(p(p-1), \ell)>1$. \end{lemma}

\begin{proof} Since $p\in P(n)$, it follows that $p\mid (n+1)!$ and $p^2\nmid (n+1)!$. Noting that
$$\Omega_\ell(n)=(n+1)! \prod_{a=1}^n \frac{a^\ell+1}{a+1}$$ is a powerful number, we see that there exists $1\le a\le n$ such that
$$p\Big| \frac{a^\ell+1}{a+1}.$$ Now Lemma \ref{lem2} follows from Lemma \ref{lem1}. \end{proof}

\begin{lemma}\label{lem4}Let $\ell$ be a positive odd number and $n$ a positive integer with $n\ge 3$.
 Suppose that there exist $k$ distinct primes $p_1, p_2, \dots , p_k\in P(n)$
 such that any two of $p_1-1, p_2-1, \dots , p_k-1$ have no common odd prime factors.
If $\Omega_\ell(n)$ is a powerful number, then $\ell$ has at least
$k$ distinct prime factors.\end{lemma}

\begin{proof}Since $(n+1)/2<p_i, p_j\le n+1$, we see that $(p_i, p_j-1)=1$ for all $1\le i,j\le k$.
Thus any two of $p_1(p_1-1), p_2(p_2-1), \dots , p_k(p_k-1)$ have no
common odd prime factors. By Lemma \ref{lem2}, we have $(p_i(p_i-1),
\ell)>1$ for all $1\le i\le k$. Since $\ell$ is odd, it follows that
$\ell$ has at least $k$ distinct prime factors.
\end{proof}

\begin{lemma}\label{lem5}Let $\ell$ be a positive odd number with at most $t$ distinct prime factors.
 Suppose that there exist $t+1$ primes $p_1, p_2, \dots ,
 p_{t+1}$ with $3\le p_1<\cdots <p_{t+1}$
 such that any two of $p_1-1, p_2-1, \dots , p_{t+1}-1$ have no common odd prime factors.
Then, for all integers $n$ with $p_{t+1}-1\le n\le
 2p_{1}-2$, $\Omega_\ell(n)$ is not a powerful number.\end{lemma}

\begin{proof}For all integers $n$ with $p_{t+1}-1\le n\le
 2p_{1}-2$, we have $p_1, p_2, \dots ,
 p_{t+1}\in P(n)$. If $\Omega_\ell(n)$ is a powerful number, then, by Lemma \ref{lem4}, $\ell$ has at least
$t+1$ distinct prime factors, a contradiction.
\end{proof}

The following two lemmas are used to deal with small values of
$n$.

\begin{lemma}\label{lem6}Let $\ell$ be a positive odd number with at most $t$ distinct prime factors.
 Suppose that there exist $(t+1)-$tuples $(p_{i,1}, p_{i,2}, \dots ,
 p_{i,t+1})$ of primes
 with $3\le p_{i,1}<p_{i,2}<\dots <p_{i,t+1}(1\le i\le s)$ and
 $p_{i+1,t+1}<2p_{i,1} (1\le i\le s-1)$
 such that, for each $i$, any two of $p_{i,1}-1, p_{i,2}-1, \dots , p_{i,t+1}-1$ have no common odd prime
 factors. Then, for all integers $n$ with $p_{1,t+1}-1\le n\le
 2p_{s,1}-2$, $\Omega_\ell(n)$ is not a powerful number.\end{lemma}

\begin{proof} Given $i$ with $1\le i\le s-1$. By Lemma \ref{lem5}, for all integers $n$ with $p_{i,t+1}-1\le n\le
 2p_{i,1}-2$, $\Omega_\ell(n)$ is not a powerful number. Since $p_{i+1,t+1}<2p_{i,1} (1\le i\le
 s-1)$, it follows that $p_{i+1,t+1}-1\le 2p_{i,1}-2 (1\le i\le
 s-1)$. Therefore, for all integers $n$ with $p_{1,t+1}-1\le n\le
 2p_{s,1}-2$, $\Omega_\ell(n)$ is not a powerful number.
\end{proof}

\begin{lemma}\label{lem7} Let $\ell$ be a positive odd integer, $p$ an odd prime and $n$ a positive integer with
$p^2\mid \Omega_\ell(n)$. Suppose that there exists $d_1\mid \ell$
with $p\| a^{d_1}+1$ for some $2\le a\le n$ and $p \nmid
b^{(\ell,p-1)}+1$ for all $2\le b\le n$ with $b\not= a$. Then
$p\mid \ell$.
\end{lemma}

\begin{proof} Suppose that $c$ is an integer with $2\le c\le n$
such that $p|c^\ell +1$. By Fermat's little theorem, we have
$p\mid c^{p-1}-1$. Hence $p\mid c^{(2\ell, p-1)}-1$. That is,
$p\mid c^{2(\ell, p-1)}-1$. By $p|c^\ell +1$ and $p\ge 3$, we have
$p\nmid c^{(\ell, p-1)}-1$. So $p\mid c^{(\ell, p-1)}+1$. By the
condition, we have $c=a$. Since $p^2\mid \Omega_\ell(n)$, it
follows that $p^2\mid a^\ell +1$. Let $\ell =d_1\ell_1$. By $p\|
a^{d_1}+1$, we have
$$p\mid \frac{a^{d_1\ell_1}+1}{a^{d_1}+1}.$$
Since
$$\frac{a^{d_1\ell_1}+1}{a^{d_1}+1}=a^{d_1(\ell_1-1)}-a^{d_1(\ell_1-2)}+\cdots
-a^{d_1}+1\equiv \ell_1 \pmod p,$$ it follows that $p\mid \ell_1$.
Therefore, $p\mid \ell$.
\end{proof}

In order to deal with large values of $n$, we need the following
lemmas.

\begin{lemma}(see \cite{MontgomeryVaughan})\label{MontgomeryVaughan} Let $x$ and $y$ be positive real numbers,
and let $k$ and $\ell$ be relatively prime positive integers. Then
$$\pi (x+y; k,\ell)-\pi (x; k,\ell)\le \frac{2y}{\varphi (k) \log (y/k)}$$
provided only that $y>k$.
\end{lemma}

\begin{lemma}(see \cite{Dusart1999})\label{Dusart1999}  For $x\ge 599$, we have
$$\frac x{\log x} \left( 1+\frac{0.992}{\log x}\right) \le \pi
(x)\le \frac x{\log x} \left( 1+\frac{1.2762}{\log x}\right).$$
\end{lemma}

\begin{lemma} (see \cite{Dusart2002})\label{Dusart12002}   For $x\ge 151$, we have
$$ \pi (x; 3, \ell )> \frac x{2\log x}.$$
\end{lemma}

\begin{lemma} (see \cite{RamareRumely}) \label{RamareRumely}  Let $k$ be an integer with $1\le k\le 72$. For any $x\ge 10^{10}$ and
any $\ell$ prime to $k$, we have
$$\left| \vartheta (x; k, \ell )- \frac x{\varphi (k)} \right| \le
0.023269 \frac x{\varphi (k)}.$$
\end{lemma}

\begin{lemma}\label{P(n)} For any integer $n$ with $n\ge 2\times  10^{10}$, we have
$$|P(n)|>0.4845 \frac {n+1}{\log (n+1)}.$$
\end{lemma}

\begin{proof} By $n\ge 2\times  10^{10}$, we have $\log ((n+1)/2)> 0.97 \log (n+1)$. By Lemma \ref{Dusart1999}, we have
\begin{eqnarray*}\pi ((n+1)/2)&\le& \frac {n+1}{2\log ((n+1)/2)} \left(
1+\frac{1.2762}{\log ((n+1)/2)}\right)\\
&<& \frac {n+1}{2\times 0.97\log (n+1)} \left(
1+\frac{1.2762}{0.97\log (n+1)}\right)\\
&\le & 0.5155 \frac {n+1}{\log (n+1)}+0.6782 \frac {n+1}{(\log
(n+1))^2}.
\end{eqnarray*}
Again, by Lemma \ref{Dusart1999}, we have
\begin{eqnarray*}|P(n)|&=&\pi (n+1)-\pi ((n+1)/2)\\
&\ge& \frac {n+1}{\log (n+1)} \left( 1+\frac{0.992}{\log
(n+1)}\right) -0.5155 \frac {n+1}{\log (n+1)}-0.6782 \frac
{n+1}{(\log
(n+1))^2}\\
&>& 0.4845 \frac {n+1}{\log (n+1)}.
\end{eqnarray*}
\end{proof}

\begin{lemma}\label{lem8} Let $5\le r\le 11$ and $n$ be an integer with $n\ge 2\times  10^{10}$. If $q$  is a prime with  $q\ge r$, then
$$|P(n; q,1)|\le \frac{n+1}{0.8696(r-1) \log (n+1)}<0.3\frac {n+1}{\log (n+1)}.$$
For $q=3$, we have
$$|P(n; 3,1)|< 0.304\frac {n+1}{\log (n+1)}.$$
\end{lemma}

\begin{proof}
 By Lemmas
\ref{Dusart1999} and \ref{Dusart12002}, we have
\begin{eqnarray*}\pi (n+1; 3,1)&=&\pi (n+1)-\pi (n+1; 3,2)\\
&\le& \frac {n+1}{\log (n+1)} \left( 1+\frac{1.2762}{\log
(n+1)}\right)-\frac {n+1}{2\log (n+1)}\\
&=&\frac {n+1}{\log (n+1)} \left( \frac 12+\frac{1.2762}{\log
(n+1)}\right)\end{eqnarray*} and
$$\pi ((n+1)/2;3,1)>\frac{n+1}{4\log ((n+1)/2)}>\frac{n+1}{4\log
(n+1)}.$$ Thus, by $n\ge 2\times  10^{10}$, we have
\begin{eqnarray*}|P(n; 3,1)|
&=&\pi (n+1; 3,1)-\pi ((n+1)/2; 3,1)\\
&<&\frac {n+1}{\log (n+1)} \left( \frac 14+\frac{1.2762}{\log
(n+1)}\right)\\
&<&0.304\frac {n+1}{\log (n+1)}.\end{eqnarray*}

Now we assume that $q\ge 5$.

Since $n\ge 2\times  10^{10}$ and $5\le r\le 11$, it follows that
\begin{equation}\label{ab}\log ((n+1)/(2r))\ge \log ((n+1)/22)\ge
0.8696 \log (n+1)\end{equation} and
\begin{equation}\label{a}0.3\frac {n+1}{\log (n+1)}>\frac{n+1}{0.8696(r-1) \log (n+1)}>6.\end{equation}

{\bf Case 1:} $n+1\ge 6q$.  By Lemma \ref{MontgomeryVaughan}, we
have
$$|P(n; q,1)|=\pi (n+1; q,1)-\pi ((n+1)/2; q, 1) \le \frac{n+1}{(q-1) \log ((n+1)/(2q))}.$$
Since the function $f(x)=(x-1)(\log ((n+1)/2)-\log x)$ is increasing
on $x\in [5, (n+1)/6]$, it follows from \eqref{ab} that
\begin{eqnarray*}|P(n; q,1)| &\le& \frac{n+1}{(q-1) \log ((n+1)/(2q))}\\
&\le& \frac{n+1}{(r-1) \log ((n+1)/(2r))}\\
&<&\frac{n+1}{0.8696 (r-1) \log (n+1)}.\end{eqnarray*}

{\bf Case 2:} $n+1< 6q$. Then, by \eqref{a}, we have
$$|P(n; q,1)| \le 6<\frac{n+1}{0.8696 (r-1) \log (n+1)}.$$
\end{proof}

\section{Proof of Theorem \ref{thm1} for $n\ge 2\times 10^{10}$}

By \cite{ChenGongRen}, we may assume that $\ell$ has exact two
distinct prime factors. Let $\ell =q_1^{\alpha_1} q_2^{\alpha_2}$,
where $q_1$ and $q_2$ are two primes with $q_1>q_2\ge 3$, and
$\alpha_1$ and $\alpha_2$ are two positive integers. In this
section, we always assume that $n\ge 2\times 10^{10}$.

{\bf Case 1:} $q_1q_2\le 72$. By Lemma \ref{RamareRumely}, we have
\begin{eqnarray*}&&\left| \vartheta (n+1; q_1q_2, 2 )-\vartheta \left( \frac12 (n+1);
q_1q_2, 2 \right) -\frac{n+1}{2\varphi(q_1q_2)}\right|\\
&\le &\left| \vartheta (n+1; q_1q_2, 2 ) -
\frac{n+1}{\varphi(q_1q_2)}\right| +\left| \vartheta \left( \frac
12 (n+1);
q_1q_2, 2 \right) - \frac{n+1}{2\varphi(q_1q_2)}\right|\\
&\le & 0.023269 \frac{n+1}{\varphi(q_1q_2)} + 0.023269
\frac{n+1}{2\varphi(q_1q_2)}\\
&<&\frac{n+1}{2\varphi(q_1q_2)}.\end{eqnarray*} Hence
$$\vartheta (n+1; q_1q_2, 2 )-\vartheta \left(\frac 12(n+1);
q_1q_2, 2 \right)>0.$$ Thus, there exists $p\in P(n)$ such that
$p\equiv 2\pmod{q_1q_2}$. Then $(p(p-1), \ell)=1$. By Lemma
\ref{lem2}, $\Omega_\ell(n)$ is not a powerful number.

{\bf Case 2:} $q_1q_2>72$. Then $q_1\ge 11$ and $q_2\ge 3$. By Lemma
\ref{lem8}, we have
$$|P(n;q_1,1)|\le \frac{n+1}{0.8696(11-1) \log (n+1)}<0.115\frac{n+1}{\log (n+1)}$$
and
$$|P(n;q_2,1)|< 0.304\frac {n+1}{\log (n+1)}.$$
Thus, by $n\ge 2\times 10^{10}$ and Lemma \ref{P(n)}, we have
$$|P(n;q_1,1)|+|P(n;q_2,1)|+2<0.419\frac {n+1}{\log (n+1)}+2<0.4845\frac {n+1}{\log
(n+1)}< |P(n)|.$$ Hence, there exists $p\in P(n)$ such that
$$p\notin \{ q_1, q_2\}, \quad p\not\equiv 1\pmod{q_1}, \quad p\not\equiv
1\pmod{q_2}.$$ For this prime $p$, we have $p\in P(n)$ and
$(p(p-1), \ell)=1$. By Lemma \ref{lem2}, $\Omega_\ell(n)$ is not a
powerful number.

\section{Proof of Theorem \ref{thm1} for $n< 2\times 10^{10}$}

Let \begin{eqnarray*}&&\{ (p_{i,1}, p_{i,2}, p_{i,3})\}_{i=1}^{7} \\
&=&\{ (17,19,23), (23,29,31), (29, 37, 41), (43, 47, 53),\\
&& (71, 73, 83), (131, 137, 139), (239, 251, 257)\} .\end{eqnarray*}
Then $p_{i,j} (1\le i\le 7, 1\le j\le 3)$ are primes such that
$p_{i,1}<p_{i,2}<p_{i,3}$, $p_{i+1,3}<2p_{i,1}$ and for each $i$,
$p_{i,1}-1, p_{i,2}-1, p_{i,3}-1$ have no common odd prime
 factors. Since
$\ell$ has exact two distinct prime factors, it follows from
 Lemma \ref{lem6} that for all integers $n$ with $22=p_{1,3}-1\le
 n\le 2p_{7,1}-2=476$, $\Omega_\ell(n)$ is not a powerful number.

Let \begin{eqnarray*}&&\{ (p_{i,1}, p_{i,2}, p_{i,3})\}_{i=8}^{34} \\
&=&\{ (263, 347, 359), (467, 479, 503),(839, 863, 887), (1487, 1523, 1619), \\
&& (2819, 2879, 2903),(5387, 5399, 5483), (10103, 10163, 10343),\\
&& (19583, 20123, 20183), (38747, 38783, 38867),  (77003, 77279, 77339), \\
&& (153563, 153743, 153887), (306563, 306707, 306899), \\
&& (611543, 611999, 612023),(1221659, 1222667, 1222967),  \\
&& (2442599, 2442767, 2443283), (4884227, 4884707, 4884779),\\
&&  (9767987, 9767999, 9768119), (19532699, 19533587, 19534583), \\
&& (39063719, 39063923, 39064643), (78126383, 78126563,
78126827),  \\
&& (156251483, 156251699, 156252623),(312501479, 312501587,
312502439),  \\
&& (625002419, 625002479, 625002527), (1250001167, 1250002847,
1250003243), \\
&& (2500001327, 2500001447, 2500001807), \\
&& (5000001299,
5000001863, 5000001983), \\
&&(10000000259, 10000000643, 10000002263 ) \}. \\
\end{eqnarray*}
Then $p_{i,j} (8\le i\le 34, 1\le j\le 3)$ are primes such that
$p_{i,1}<p_{i,2}<p_{i,3}$, $p_{i+1,3}<2p_{i,1}$ and for each $i$,
$(p_{i,1}-1)/2, (p_{i,2}-1)/2, (p_{i,3}-1)/2$ are distinct primes.
Since $\ell$ has exact two distinct prime factors, it follows from
 Lemma \ref{lem6} that for all integers $n$ with $358=p_{8,3}-1\le
 n\le 2\times 10^{10}<
 2p_{34,1}-2=20000000516$, $\Omega_\ell(n)$ is not a powerful number.

It is clear that $\Omega_\ell(n)$ is not a powerful number for
$n\in \{ 1, 2\}$. Finally, we consider $3\le n\le 21$. Suppose
that $\Omega_\ell(n)$ is a powerful number.

{\bf Case 1:} $16\le n\le 21$. Then $13,17\in P(n)$. By Lemma
\ref{lem2}, we have $17\mid \ell$ and $(3\times 13, \ell)>1$. So
$\ell =3^\alpha 17^\beta$ or $13^\alpha 17^\beta$. By calculation,
we have $953\| 4^{17}+1$ and $953\nmid b^{17}+1$ for all $1\le
b\le 21$ with $b\not= 4$. Since $(\ell, 953-1)=17$, it follows
from Lemma \ref{lem7} that $953\mid \ell$, a contradiction.

{\bf Case 2:} $10\le n\le 15$. Then $11\in P(n)$. By Lemma
\ref{lem2}, we have $(5\times 11, \ell)>1$.

If $5\mid \ell$, then, noting that $41\|  4^5+1$, $(41-1,\ell )=5$
and $41\nmid b^5+1$ for all $1\le b\le 15$ with $b\not= 4$, by
Lemma \ref{lem7} we have $41\mid \ell$. Since $\ell$ has exact two
distinct prime factors, it follows that $\ell =5^\alpha 41^\beta$.
Thus $(61-1, \ell)=5$. Noting that $61\|  3^5+1$, $61\nmid b^5+1$
for all $1\le b\le 15$ with $b\not= 3$, by Lemma \ref{lem7} we
have $61\mid \ell$, a contradiction. Hence $5\nmid \ell$ and
$11\mid \ell$. Noting that $661\|  3^{11}+1$, $(661-1,\ell)\in \{
11, 3\times 11\}$, $661\nmid b^{11}+1$ and $661\nmid b^{33}+1$ for
all $1\le b\le 15$ with $b\not= 3$, by Lemma \ref{lem7} we have
$661\mid \ell$. Thus $\ell =11^\alpha 661^\beta$. Then $(397-1,
\ell)=11$. Noting that $397\| 4^{11}+1$ and $397\nmid b^{11}+1$
for all $1\le b\le 15$ with $b\not= 4$, by Lemma \ref{lem7} we
have $397\mid \ell$, a contradiction.

{\bf Case 3:} $6\le n\le 9$. Then $7\in P(n)$.  By Lemma
\ref{lem2}, we have $(3\times 7, \ell)>1$.

If $3\mid \ell$, then, noting that $13\| 4^3+1$, $(13-1, \ell)=3$
and $13\nmid b^3+1$ for all $1\le b\le 9$ with $b\not= 4$, by
Lemma \ref{lem7} we have $13\mid \ell$. Thus $\ell =3^\alpha
13^\beta$. Noting that $31\| 6^{3}+1$, $(31-1,\ell)=3$, and
$31\nmid b^{11}+1$ for all $1\le b\le 9$ with $b\not= 6$, by Lemma
\ref{lem7} we have $31\mid \ell$, a contradiction. Hence $3\nmid
\ell$ and $7\mid \ell$. Noting that $449\| 5^{7}+1$,
$(449-1,\ell)=7$, and $449\nmid b^{7}+1$ for all $1\le b\le 9$
with $b\not= 5$, by Lemma \ref{lem7} we have $449\mid \ell$. Thus
$\ell =7^\alpha 449^\beta$. Then $(547-1, \ell)=7$.  Noting that
$547\| 3^{7}+1$ and $547\nmid b^{7}+1$ for all $1\le b\le 9$ with
$b\not= 3$, by Lemma \ref{lem7} we have $547\mid \ell$, a
contradiction.

{\bf Case 4:} $4\le n\le 5$. Then $5\in P(n)$.  By Lemma
\ref{lem2}, we have $5\mid \ell$. By calculation, we have $11\|
2^5+1$, $41\| 4^{5}+1$, $11\nmid b^5+1$ for all $1\le b\le 5$ with
$b\not= 2$, and $41\nmid c^{5}+1$ for all $1\le c\le 5$ with
$c\not= 4$. Since $(\ell, 11-1)=5$ and $(\ell, 41-1)=5$, it
follows from Lemma \ref{lem7} that $11\times 41\mid \ell$, a
contradiction.

{\bf Case 5:} $n=3$. Since $3\mid 2^\ell +1$, $3\nmid 3^\ell+1$
and $\Omega_\ell(n)$ is a powerful number, we see that $3^2\mid
2^\ell+1$. So $3\mid \ell$. It is clear that $(7-1, \ell)=3$. By
$7\nmid 2^3+1$, $7\| 3^3+1$ and Lemma \ref{lem7}, we have $7\mid
\ell$. Thus $\ell=3^\alpha 7^\beta$. Then $(\ell, 547-1)=21$.
Noting that $547\| 3^{7}+1$ and $547\nmid 2^{21}+1$, by Lemma
\ref{lem7} we have $547\mid \ell$, a contradiction.

\section{Proof of Theorem \ref{thm2}}

To prove Theorem \ref{thm2}, we first give several lemmas. In this
section, the bounds are not best possible.

\begin{lemma}\label{lema} Let $q$ be an odd prime with $q\ge 16C+1\ge 3$.
If $n\ge 4(16C+1)^2$, then
$$|P(n; q, 1)|\le \frac{n+1}{8C\log (n+1)}+6.$$   \end{lemma}

\begin{proof} If $n+1\ge 6q$, then, by Lemma \ref{MontgomeryVaughan},
we have $$|P(n; q, 1)|=\pi (n+1; q,1)-\pi (\frac12 (n+1); q,1)\le
\frac{n+1}{(q-1)\log ((n+1)/(2q))}.$$ Since the function
$f(x)=(x-1)(\log ((n+1)/2)-\log x)$ is increasing on $x\in [3,
(n+1)/6]$, it follows from $q\ge 16C+1$ and $n\ge 4(16C+1)^2$ that
\begin{eqnarray*}&&|P(n; q, 1)|\le \frac{n+1}{(q-1)\log ((n+1)/(2q))}\\
&\le& \frac{n+1}{16C\log ((n+1)/(2(16C+1)))}\le \frac{n+1}{8C\log
(n+1)}.\end{eqnarray*} If $n+1< 6q$, then $|P(n; q, 1)|\le \pi (n+1;
q,1)\le 6$.
\end{proof}

\begin{lemma}(Dirichlet's Theorem (see \cite{Davenport})\label{lemc} We have
$$\pi (x; k,\ell) \sim \frac{x}{\varphi (k)\log x}$$. \end{lemma}

\begin{lemma}\label{lemd} For any odd integer $k\ge 3$, there exists
an integer $M(k)$ such that, if $n$ is an integer with $n\ge
M(k)$, then
$$|P(n; k,2)|\ge \frac{n+1}{4\varphi (k) \log (n+1)}.$$\end{lemma}

Lemma \ref{lemd} follows from Lemma \ref{lemc} and $$|P(n;
k,2)|=\pi (n+1; k,2)-\pi ((n+1)/2; k,2).$$

\begin{proof}[Proof of Theorem \ref{thm2}] It is clear that we
may assume that $\ell$ has exact $r$ distinct prime factors.  Let
$\ell =q_1^{\alpha_1} \cdots q_r^{\alpha_r}$, where $q_i (1\le
i\le r)$ are primes with $3\le q_1<\cdots <q_r$ and $\alpha_i
(1\le i\le r)$ are positive integers. Let
$$D_1=16r+1, \quad D_{i+1}=16r\cdot D_i!!+1,\quad i=1,2,\dots ,
r-1,$$ where $D_i!!=\prod_{0\le k\le (D_i-1)/2} (2k+1) (1\le i\le
r)$.

If $q_1\ge D_1$, then, by Lemmas \ref{lema} and \ref{P(n)}, we have
$$\sum_{i=1}^r |P(n; q_i,1)|+r\le  \frac{n+1}{8\log (n+1)}+7r<0.4845 \frac {n+1}{\log
(n+1)}<|P(n)|$$ for $n\ge N_{1,r}$. Thus, there exists a prime $p\in
P(n)$ such that $p\notin \{ q_1, \dots , q_r\}$ and $p\not\equiv
1\pmod{q_i} (1\le i\le r)$. For this prime $p$, we have $p\in P(n)$
and $(p(p-1), \ell)=1$. By Lemma \ref{lem2}, $\Omega_\ell (n)$ is
not a powerful number.

Now we assume that $q_1<D_1$. Suppose that $q_i<D_i$ for $1\le
i<j\le r$ and $q_j\ge D_j$. By Lemmas \ref{lema} and \ref{lemd}, we
have \begin{eqnarray*}&&\sum_{i=j}^r |P(n; q_i,1)|+r\le
\frac{n+1}{8\cdot D_{j-1}!!\log (n+1)}+7r \\
&<&\frac{n+1}{4\varphi (D_{j-1}!!) \log (n+1)}\le |P(n; D_{j-1}!!,
2)|\end{eqnarray*} for $n\ge N_{j,r}$. Thus, there exists a prime
$p\in P(n; D_{j-1}!!, 2)$ such that $p\notin \{ q_1, \dots ,
q_r\}$ and $p\not\equiv 1\pmod{q_i} (j\le i\le r)$. By $q_1\cdots
q_{j-1}\mid D_{j-1}!!$, we have $p\equiv 2\pmod{q_1\cdots
q_{j-1}}$. Hence, for this prime $p$,  we have $p\in P(n)$ and
$(p(p-1), \ell)=1$. By Lemma \ref{lem2}, $\Omega_\ell (n)$ is not
a powerful number.

Finally, suppose that $q_i<D_i$ for $1\le i\le r$. By Lemma
\ref{lemd}, we have
$$|P(n; D_r!!,2)|\ge \frac{n+1}{4\varphi (D_r!!) \log (n+1)}>1$$
for $n\ge N_{r+1,r}$. Thus, there exists a prime $p\in P(n; D_{r}!!,
2)$. Since $q_1\cdots q_r
 \mid D_r!!$, it follows that $p\equiv 2\pmod{q_1\cdots q_r}$. So
 $(p(p-1), \ell)=1$. By Lemma \ref{lem2}, $\Omega_\ell (n)$ is not a powerful
number.

Let $T_r=\max \{ N_{1,r},\dots , N_{r+1, r}\} $. We have proved that
if $n\ge T_r$, then $\Omega_\ell (n)$ is not a powerful number. This
completes the proof of Theorem \ref{thm2}.
\end{proof}

\end{document}